\documentclass{article}

\usepackage{amsmath}
\usepackage[latin1]{inputenc}
\usepackage[T1]{fontenc}
\usepackage[english]{babel}
\usepackage{enumerate}
\usepackage{amssymb}
\usepackage{amsthm}
\usepackage{cleveref}
\usepackage{url}

\newtheorem{theorem}{Theorem}[section]
\newtheorem{Lemma}[theorem]{Lemma}
\newtheorem{Theorem}[theorem]{Theorem}

\newtheorem*{Proposition*}{Proposition}
\newtheorem*{Conjecture*}{Conjecture}
\newtheorem*{Theorem*}{Theorem}

\title{The Hyperrigidity Conjecture for Spectrahedra}
\author{Marcel Scherer}
\date{}

\begin{document}

\maketitle
\begin{abstract}
We show that if $K$ is a compact spectrahedron whose set of extreme points is closed, then the operator system of continuous affine functions on $K$ is hyperrigid in the $C^*$-algebra $C(\textup{ex}(K))$. 
\end{abstract}
\let\thefootnote\relax\footnotetext{ \date
}

\section{Introduction}
An \textit{operator system} $S$ is a unital, $*$-closed subspace of a $C^*$-algebra $\mathcal{A}$. A unital completely positive map, for short \textit{u.c.p.}\ map, $\phi:S\to\mathcal{B}(H)$, where $H$ is a Hilbert space, is said to have the \textit{unique extension property} (u.e.p.) if every u.c.p.\ extension of $\phi$ to $C^*(S)$ is a $*$-homomorphism. Note that, by Arveson's extension theorem (see \cite{Pa}), every u.c.p.\ map on $S$ admits at least one u.c.p.\ extension to $C^*(S)$.\\
The study of the unique extension property has proved very fruitful, and in this paper we will study \textit{hyperrigid operator systems}, i.e., separable operator systems such that for every unital $*$-homomorphism $\pi$ on $C^*(S)$ the restriction $\pi|_S$ has the u.e.p.. Such operator systems were introduced by Arveson in \cite{Arv1}, however, note that Arveson's original definition is phrased in approximation-theoretic terms, and the equivalence with the formulation above is proved in \cite{Arv1}. \\
Establishing hyperrigidity is often subtle, since one must control all $*$-representations of $C^*(S)$. A remedy would follow from Arveson's hyperrigidity conjecture:

\begin{Conjecture*}[Hyperrigidity Conjecture]
A separable operator system $S$ is hyperrigid if and only if the restriction of every irreducible representation of $C^*(S)$ to $S$ has the u.e.p..
\end{Conjecture*}

In the following, we say that an operator system $S$ \textit{satisfies the conjecture} if, whenever the restriction of every irreducible representation of $C^*(S)$ to $S$ has the u.e.p., then $S$ is hyperrigid. \\
It has recently been shown that the conjecture fails in the stated generality, see \cite{BiDo, BB, Sch2, Clo}, but as a glimpse of hope, all currently known counterexamples share the feature that $C^*(S)$ is non-commutative. This motivates restricting attention to operator systems inside commutative $C^*$-algebras, also known as \textit{function systems}.\\
This paper is far from the first to pursue this point of view. We briefly review existing results of hyperrigid operator systems in commutative $C^*$-algebras and highlight some advantages of working in this setting.\\
First, if $C^*(S)$ is commutative, then we can identify it with $C(X)$ for some compact Hausdorff space $X$, and hence every irreducible representation of $C^*(S)$ is an evaluation map
  \[
    e_x:C(X)\to\mathbb{C},\ f\mapsto f(x),\ x\in X.
  \]
Second, set
  \[
    K=\{\phi:S\to\mathbb{C};\ \phi\ \textup{is a state}\}.
  \]
Then $K$ is a compact convex set, and we can consider $A(K)$, the space of continuous affine functions on $K$. Denoting the extreme points of $K$ by $\textup{ex}(K)$, we regard $A(K)$ as a subspace of $C(\overline{\textup{ex}(K)})$, thereby turning it into an operator system. The canonical map
  \[
    S\to A(K), s\mapsto [\phi\mapsto \phi(s)]
  \]
is a unital completely isometric map with dense image. In this way, questions about hyperrigidity of $S$ can be translated into questions about the function system $A(K)$.
\begin{Proposition*}
$S$ is hyperrigid in its $C^*$-envelope $C^*_e(S)$ if and only if $A(K)$ is hyperrigid in $C(\overline{\textup{ex}(K)})$.
\end{Proposition*}
Since the $C^*$-envelope of an operator system will not be relevant beyond this proposition, we do not go further into detail. For our purposes it suffices to note that if $S$ is hyperrigid in $C^*(S)$, then $C^*(S)=C^*_e(S)$.\\
This new viewpoint connects hyperrigidity to classic convex-geometric concepts, for example:
\begin{Proposition*}
The restriction of every irreducible representation of $C(\overline{\textup{ex}(K)})$ to $A(K)$ has the u.e.p. if and only if $\overline{\textup{ex}(K)}=\textup{ex}(K)$.
\end{Proposition*}
This perspective was exploited, for instance, in \cite{DaKe}, where the classical Choquet order was generalized to the dilation order, and in \cite{Sch1}, where it was shown that $A(K)$ is hyperrigid in $C(\textup{ex}(K))$ for every compact convex set $K\subset\mathbb{R}^2$ (in this case $\textup{ex}(K)$ is automatically closed).\\
Further examples in commutative $C^*$-algebras were obtained in \cite{PaSt}: if $f\in C(X)$ separates the points of $X$ and $p,q,r\in\mathbb{Z}_+$ satisfy $p\neq q$ and $p+q<2r$, then the operator system generated by
  \[
    \{|f|^{2r}, \bar{f}^pf^q\}
  \]
is hyperrigid in $C(X)$.\\
Another general source of hyperrigid examples comes from the multiplicative domain: for a u.c.p. $\phi$ on a $C^*$-algebra $\mathcal{A}$, the set
  \[
    \{x\in\mathcal{A};\ \phi(x^*)\phi(x)=\phi(x^*x)\ \textup{and}\ \phi(x)\phi(x^*)=\phi(xx^*)\}
  \]
is a $C^*$-algebra (see \cite{Pa}). Thus, if $\mathcal{A}$ is generated by $x_1,\dots,x_n$, then the operator system generated by
  \[
    \{x_i;\ 1\le i\le n\}\cup\{x_i^*x_i;\ 1\le i\le n\}\cup\{x_ix_i^*;\ 1\le i\le n\}
  \]
is hyperrigid in $\mathcal{A}$. Variants of this observation recover the basic examples in \cite{Arv1} and, for instance, yield the hyperrigidity of $A(\overline{\mathbb{B}_d})$ in $C(S^{d-1})$, where $\mathbb{B}_d$ is the unit ball in $\mathbb{R}^d$ and $S^{d-1}$ the sphere.\\ 
We conclude our list of known examples of hyperrigid operator systems in commutative $C^*$-algebras with two observations. First, if $X$ is countable, then every operator system in $C(X)$ satisfies the conjecture by \cite{Arv1}. Second, if $S$ is hyperrigid in $C^*(S)$, then every operator system $\tilde{S}$ in $C^*(S)$ with $S\subset\tilde{S}$ is hyperrigid as well.\\
In this paper, we add spectrahedra to the list of convex sets that give rise to hyperrigid function systems. Recall that a spectrahedron is a set of the form
  \[
    \left\{(z_1,\dots,z_k)\in\mathbb{R}^k;\ Q(z)=A_0+\sum_{i=1}^kz_iA_i\ge0\right\},
  \]
where $A_0, A_1,\dots, A_k$ are selfadjoint matrices. Such a set need not be compact, so throughout the paper we assume compactness. The following theorem is Theorem \ref{Thm: main}.
  \begin{Theorem*}
  Let $K$ be a compact spectrahedron such that $\textup{ex}(K)$ is closed. Then $A(K)$ is hyperrigid in $C(\textup{ex}(K))$.
  \end{Theorem*}
The proof proceeds by constructing positive elements in $M_n(A(K))$, the $n\times n$ matrices with entries in $A(K)$, from positive semidefinite kernels associated to the defining pencil $Q(\cdot)$. Concretely, on suitable subsets of $\partial K$, we select a normalized map $\gamma$ with $\gamma(z)\in\ker(Q(z))$ and use it to build kernels of the form
  \[
    (\omega,z)\mapsto \gamma(z)^*Q(x)\gamma(\omega), \qquad x\in K,
  \]
see Section \ref{Sec: Spectra}. From these kernels we obtain positive elements $A(\cdot)\in M_n(A(K))$. A key step is then to dominate characteristic matrices via a Hadamard product estimate
  \[
    (\chi_V)_{1\le i,j\le n}\le A(\cdot)\circ B,
   \]
where $V$ is a suitable subset of $\partial K$ and $B$ is a positive matrix constructed using \cite{Rea} with quantitative Perron-Frobenius bounds. The required Perron eigenvalue/eigenvector estimates and the necessary background on this theory are given in the following Section \ref{Sec: Perron}.

\section{Perron Eigenvalues and Eigenvectors}\label{Sec: Perron}

Let $A\in M_n$ be a matrix with strictly positive entries. By the Perron-Frobenius theorem, the spectral radius
  \[
    \rho(A)=\max\{|\mu|;\ \mu\in\sigma(A)\}
  \]
is a simple eigenvalue of $A$ and admits a unique normalized eigenvector $x=(x_1,\dots,x_n)\in\mathbb{R}^n$, the \textit{normalized Perron eigenvector}, such that $x_i>0$ for all $i$. We call $\rho(A)$ the \textit{Perron eigenvalue} of $A$. In particular, if $A$ is selfadjoint, then $\rho(A)=\|A\|$. \\

For the proof of the main theorem, we need the following bounds on the Perron eigenvalue and on the entries of the normalized Perron eigenvector.

\begin{Lemma}\label{Lem: Perron Bound}
Let $A=(a_{i,j})_{1\le i,j\le n}$ be selfadjoint with strictly positive entries and assume that there exist constants $0<c\le b$ such that 
  \[
    c\le a_{i,j}\le b, \qquad 1\le i,j\le n.
  \]    
Let $x$ be the normalized Perron eigenvector and $\lambda$ the Perron eigenvalue of $A$. Then, for all $1\le i\le n$,
  \[
    \tfrac{c}{b\sqrt{n}}\le x_i\le \tfrac{b}{c\sqrt{n}},
  \]
and 
  \[
    nc\le \lambda.
  \]
In particular, for all $1\le i,j\le n$,
  \[
    \tfrac{c^3}{b^2}\le\langle\lambda xx^* e_i,e_j\rangle
  \]
  for all $1\le i,j\le n$.
\end{Lemma}
  
\begin{proof}
With the above notations, let $x_l=\textup{min}_j\{x_j\}$ and $x_k=\textup{max}_j\{x_j\}$. Since $Ax=\lambda x$ and $c\le a_{k,i}\le b$, we have
   \[
     \lambda x_k=\sum_{i=1}^n a_{k,i}x_i\le b\|x\|_1
  \]
  and
  \[
    \lambda x_l=\sum_{i=1}^n a_{l,i}x_i\ge c\|x\|_1.
  \]
Therefore,
  \[ 
     \lambda x_k\le b\|x\|_1= \tfrac{b}{c} c\|x\|_1\le\tfrac{b}{c} \lambda x_l,
  \]
so $x_k\le\tfrac{b}{c}x_l$. Using $\|x\|_2=1$ we obtain
     \[
      \tfrac{nc^2}{b^2}x_k^2\le n x_l^2\le1\le n x_k^2\le \tfrac{nb^2}{c^2} x_l^2, 
  \]
and thus
  \[
    \tfrac{c}{b\sqrt{n}}\le x_l\le x_k\le \tfrac{b}{c\sqrt{n}}. 
  \]
For the eigenvalue bound, let $\mathbf{1}=(1,\dots,1)^T$. Since each entry of $A\mathbf{1}$ is at least $nc$, we have
  \[
    \|A\mathbf{1}\|_2\ge nc\|\mathbf{1}\|_2,
  \]
and therefore
  \[
    nc\le \tfrac{\|A\mathbf{1}\|_2}{\|\mathbf{1}\|_2}\le \|A\|=\lambda.
  \]
Finally, $\langle \lambda xx^*e_i,e_j\rangle=\lambda x_ix_j\ge(nc)\left(\tfrac{c}{b\sqrt{n}}\right)^2=\tfrac{c^3}{b^2}$.
\end{proof}

\section{Spectrahedra}\label{Sec: Spectra}

Let $\textup{Sym}_m$ be the set of real symmetric $m\times m$ matrices, and let $\textup{Sym}_m^+\subset\textup{Sym}_m$ be the cone of positive semidefinite matrices. Given selfadjoint matrices $A_0,\dots, A_k$ in $M_m$, the map
  \[
    Q:\mathbb{R}^k\to M_m, (z_1,\dots,z_k)\mapsto A_0+\sum_{i=1}^k z_i A_i
  \]
is called a \textit{matrix pencil} and the set
  \[
    K=\{z\in\mathbb{R}^k;\ Q(z)\ge0\}
  \]
a \textit{spectrahedron}. If the matrices $A_0,\dots, A_k$ are in $\textup{Sym}_m$, then the pencil $Q$ is called \textit{symmetric}. It is clear that $K$ is convex and closed, however, not necessarily compact or non-empty. Since these cases are not of interest in terms of hyperrigidity, we will assume throughout the paper that $K$ is compact and non-empty. \\
It is well-known that for every spectrahedron $K$ there exists a symmetric pencil $\tilde{Q}$ such that 
  \[
    K=\{z\in\mathbb{R}^k;\ \tilde{Q}(z)\ge0\}.
  \]
Indeed, let $A_0,\dots, A_k$ be selfadjoint and $Q$ be as above. Let $A_i=X_i+iY_i$ with $X_i=\textup{Re}(A_i)$ and $Y_i=\textup{Im}(A_i)$ (entrywise), so $X_i, Y_i\in M_m(\mathbb{R})$. Define the symmetric pencil
  \[
    \tilde{Q}(z)=\begin{pmatrix} X_0 & -Y_0\\ Y_0 & X_0 \end{pmatrix}+\sum_{i=1}^k z_i\begin{pmatrix} X_i & -Y_i\\ Y_i & X_i \end{pmatrix}.
  \]
Then $Q(z)\ge0$ if and only if $\tilde{Q}(z)\ge0$, hence
  \[
    K=\{z\in\mathbb{R}^k;\ Q(z)\ge0\}=\{z\in\mathbb{R}^k;\ \tilde{Q}(z)\ge0\}.
  \]
So we may and will assume for the rest of the paper that the pencil $Q$ is symmetric.\\
Since the cone of strictly positive definite matrices is open, we have
  \[
    \ker(Q(z))\neq\{0\}
  \]
for all $z\in\partial K$. Thus there exists a (not necessarily continuous) map
  \[
    \gamma:\partial K\to \{z\in\mathbb{R}^m;\ \|z\|_2=1\}
  \]
such that $\gamma(z)\in\ker(Q(z))$ for all $z\in\partial K$. For $z\in\mathbb{R}^k$ define
  \[
    k^\gamma_z:\partial K\times \partial K\to\mathbb{R}, (x,y)\mapsto \gamma(y)^*Q(z)\gamma(x).
  \]
 It is straightforward to check that for every $z\in K$ the kernel $k^\gamma_z$ is positive semidefinite. Hence, given $t_1,\dots,t_n\in\partial K$, the map
   \[
    A:K\to M_n, z\mapsto (k^\gamma_z(t_i,t_j))_{1\le i,j\le n},
  \]
is a continuous affine function with $A(z)\ge0$ for all $z\in K$. Moreover, if $x, y\in\partial K$, then
  \[
    k^\gamma_x(x,y)=0\qquad \textup{and}\qquad k^\gamma_y(x,y)=0.
  \]
For our purposes we will need $\gamma$ to be continuous, however, in general one cannot expect to find a continuous choice on all of $\partial K$. So we decompose the boundary according to the dimension of the kernel. Set
  \[
    K_i=\{z\in\partial K;\ \dim \ker(Q(z))=i\},\qquad i=1,\dots,m.
  \]
Then $\partial K=\bigcup_{i=1}^mK_i$ and the sets $K_i$ are pairwise disjoint. Note that if $K_m\neq\emptyset$ and $z\in K_m$, then $Q(z)=0$, so for every $\omega\in K$ and every $t\le 1$,
  \[
    Q(\omega+t(z-\omega))=A_0+\sum_{i=1}^k(\omega_i+t(z_i-\omega_i))A_i=Q(\omega)-tQ(\omega)=(1-t)Q(\omega)\ge0
  \]
Thus $\omega+t(z-\omega)\in K$ for all $t\le 1$, and compactness of $K$ forces $K=\{z\}$.

\begin{Lemma}\label{Lem: empty ker}
For $x\in\textup{ex}(K)$, let $P\in \mathcal{B}(\mathbb{R}^m)$ be the projection onto $\ker(Q(x))$. Then
  \[
    \{x\}=K\cap \{z\in \mathbb{R}^k;\ PQ(z)P=0\}.
  \]
\end{Lemma}

\begin{proof}
Fix $x\in\textup{ex}(K)$ and let $P$ be the projection onto $\ker(Q(x))$. Set
  \[
    F_x=\{z\in K;\ PQ(z)P=0\}. 
  \]    
Since $Q(z)\ge0$ for $z\in K$, we have 
  \[
    F_x=\{z\in K;\ PQ(z)P=0\}=\{z\in K;\ \ker(Q(x))\subset \ker(Q(z))\}.
  \]
We next show that $F_x$ is a face of $K$. Let $z_1, z_2\in K, \omega\in F_x, s\in (0,1)$ such that $\omega= sz_1+(1-s)z_2\in F_x$. Then, using that $Q(\cdot)$ is affine,
  \[ 
    0=PQ(\omega)P=sPQ(z_1)P+(1-s)PQ(z_2)P.
  \]
Both summands are positive semidefinite, hence each must be $0$, i.e., $z_1, z_2\in F_x$. Thus $F_x$ is a face of $K$.\\
 Let 
  \[
    C=A_0+\textup{span}\{A_1,\dots,A_k\},
  \]    
an affine subspace of $\textup{Sym}^m$, and define
  \[
    L:K\to\textup{Sym}_m^+, z\mapsto Q(z).
  \]
Then $L$ is affine and $L(K)=\textup{Sym}_m^+\cap C$. Moreover, $L$ is injective (this follows from the assumption that $K$ is compact).\\
Since $L$ is an affine bijection from $K$ to $L(K)$, it maps faces to faces and preserves smallest faces. In particular, $L(F_x)$ is a face of $L(K)$ and
  \[
    L(\textup{face}_K(x))=\textup{face}_{L(K)}(Q(x)),
  \]
where $\textup{face}_X(Y)$ denotes the smallest face of $X$ containing $Y$.\\
It is well-known (see \cite[Proposition 10.1.2]{FrMo}) that all faces of $\textup{Sym}_+^m$ are of the form 
  \[
    \{A\in\textup{Sym}^+_m;\ U\subset\ker(A)\}
  \]
for some subspace $U\subset\mathbb{R}^m$. Hence
  \[
    \textup{face}_{\textup{Sym}_m^+}(Q(x))=\{A\in\textup{Sym}_m^+;\ \ker(Q(x))\subset\ker(A)\}.
  \]
By \cite[Proposition 1]{MaSt},
  \[
    \textup{face}_{\textup{Sym}_m^+\cap C}(Q(x))=\textup{face}_{\textup{Sym}_m^+}(Q(x))\cap C.
  \] 
Since $\textup{Sym}_m^+\cap C=L(K)$, it follows that
  \[
    \begin{split}
      \textup{face}_{L(K)}(Q(x))&=\textup{face}_{\textup{Sym}_m^+}(Q(x))\cap C=\{A\in \textup{Sym}_m^+;\ \ker(Q(x))\subset\ker(A)\}\cap C\\
      &=\{A\in L(K);\ \ker(Q(x))\subset\ker(A)\}\\
      &=L(\{z\in K;\ \ker(Q(x))\subset\ker(Q(z))\})=L(F_x).
    \end{split}
  \]
Therefore,
  \[
    L(\textup{face}_K(x))=\textup{face}_{L(K)}(Q(x))=L(F_x). 
  \]    
By injectivity of $L$ we conclude that $\textup{face}_K(x)=F_x$. Since $x$ is an extreme point of $K$, we have $\textup{face}_K(x)=\{x\}$ and hence $F_x=\{x\}$, as claimed. 
\end{proof}

Next we show that each $K_i$ is \textit{locally closed} in $\mathbb{R}^k$, i.e., the intersection of a closed and an open subset of $\mathbb{R}^k$.

\begin{Lemma}\label{Lem: loc clo}
For each $i\in\{1,\dots,m\}$ the set
  \[
    K_i=\{z\in\partial K;\ \dim\ker(Q(z))=i\}
  \]
is locally closed in $\mathbb{R}^k$. In particular, for every $x\in K_i$ there exists a closed set $F\subset\mathbb{R}^k$ such that $x\in F\subset K_i$, and there exists a neighborhood $U$ of $x$ such that $U\cap K_i\subset F$.
\end{Lemma}

\begin{proof}
For $j=1,\dots,m$ set
  \[
    E_j=\bigcup_{n=j}^mK_n=\{z\in\partial K;\ \dim \ker(Q(z))\ge j\}.
  \]
Since $Q(\cdot)$ is continuous and $\partial K$ is closed, each $E_j$ is closed. Thus $K_i$ is the intersection of an open and a closed set:
    \[
      K_i=E_i\cap (\mathbb{R}^k\setminus E_{i+1}),
  \]
implying that $K_i$ is locally closed.\\ 
For the additional claim, fix $x\in K_i$ and let $\epsilon>0$ such that $B_\epsilon(x)\subset (\mathbb{R}^k\setminus E_{i+1})$. Set
  \[
    F=E_i\cap \overline{B_{\epsilon/2}(x)}.
  \]
Then $F$ is closed in $\mathbb{R}^k$, and
  \[
    F\subset E_i\cap(\mathbb{R}^k\setminus E_{i+1})=K_i.
  \]
Moreover, 
  \[
    B_{\epsilon/2}(x)\cap K_i\subset F.
  \]

  \end{proof}

\begin{Lemma}\label{Lem: cont gamma}
For every $i\in\{1,\dots,m\}, x\in K_i\cap\textup{ex}(K)$, and $y\in K\setminus\{x\}$, there exists an open set $U\subset\mathbb{R}^k$ with $x\in U$ and a continuous map $\gamma:U\cap K_i\to\mathbb{R}^m$ such that
  \begin{enumerate}[(i)]
    \item $\gamma(z)\in\ker(Q(z))$ for all $z\in U\cap K_i$,
    \item $\|\gamma(z)\|_2=1$ for all $z\in U\cap K_i$,
    \item $0< \langle Q(y)\gamma(z),\gamma(\omega)\rangle$ for all $\omega, z\in U\cap K_i$.
  \end{enumerate}
\end{Lemma}

\begin{proof}
We know that if $K_m\neq\emptyset$, then $K=\{z\}$ for some $z\in\mathbb{R}^k$ and there is nothing to show. Hence we may assume $K_m=\emptyset$. \\
Fix $i\in\{1,\dots,m-1\}$ and let $x\in K_i\cap\textup{ex}(K)$.
Write $\lambda_1(z)\le\cdots\le\lambda_m(z)$ for the eigenvalues of $Q(z)$.
Since $Q(\cdot)$ is  continuous and symmetric, the functions
$z\mapsto\lambda_k(z)$ are continuous, see \cite[Section 1.3.3]{Tao}. Because $x\in K_i$ and $Q(x)\ge0$, we have
\[
  \lambda_1(x)=\cdots=\lambda_i(x)=0,\qquad \lambda_{i+1}(x)>0.
\]
Choose $\varepsilon>0$ with $\varepsilon<\lambda_{i+1}(x)$.
By continuity, there is an open neighborhood
$U_0$ of $x$ such that for all $z\in U_0$ we have
\[
  |\lambda_i(z)|<\varepsilon
  \quad\text{and}\quad
  \lambda_{i+1}(z)>\varepsilon.
\]
In particular, $0$ is separated from the rest of the spectrum of $Q(z)$ for every
$z\in U_0$. Hence the Riesz spectral projection
\[
  P(z)=
  \frac{1}{2\pi i}\int_{|w|=\varepsilon}(wI-Q(z))^{-1}\,dw
\]
is well-defined and depends continuously on $z\in U_0\cap K_i$. Moreover, for $z\in U_0\cap K_i$, we have
\[
  \mathrm{im}\,P(z)=\ker(Q(z)),\qquad \operatorname{rank}P(z)=i,
\]
since in this case the only eigenvalue of $Q(z)$ in the disk $B_\epsilon(0)$ is $0$ with multiplicity $i$.\\
Let $P$ be the projection onto $\ker(Q(x))$. By Lemma \ref{Lem: empty ker}, we have $PQ(y)P\neq 0$. Thus there exists a $0\neq v\in\ker(Q(x))$ such that
  \[
    0<\langle Q(y)v,v\rangle.
  \]
Define
\[
  g(z)=P(z)v,\qquad z\in U_0.
\]
Then $g$ is continuous and $g(x)=P(x)v=v\neq0$. Therefore, there exists an open neighborhood $U\subset U_0$ of $x$ such that $g(z)\neq0$ for all $z\in U$ and
  \[
    0<\langle Q(y)g(z),g(\omega)\rangle\qquad (z,\omega\in U),
  \]
by continuity of $g$.\\
Finally, for $z\in U\cap K_i$ set
  \[
    \gamma(z)=\tfrac{g(z)}{\|g(z)\|_2}.
  \]
Then $\gamma$ is continuous on $U\cap K_i$, satisfies $\|\gamma(z)\|_2=1$, and since $g(z)\in\textup{im} P(z)=\ker(Q(z))$ for $z\in U\cap K_i$, we obtain $\gamma(z)\in\ker(Q(z))$. Moreover, for all $z, \omega\in U\cap K_i$,
  \[
    \langle Q(y)\gamma(z),\gamma(\omega)\rangle=\frac{\langle Q(y)g(z),g(\omega)\rangle}{\|g(z)\|_2\|g(\omega)\|_2} > 0
  \]
\end{proof}

\begin{Lemma}\label{Lem: cont gamma 2}
For every $i\in\{1,\dots,m\}, x\in K_i\cap\textup{ex}(K)$ and $y\in K\setminus\{x\}$, there exists a subset $F\subset K_i$ which is closed in $\mathbb{R}^k$ and a neighborhood $U$ of $x$ such that $U\cap K_i\subset F$, and a continuous map $\gamma:F\to\mathbb{R}^m$ such that
  \begin{enumerate}[(1)]
    \item $\gamma(z)\in\ker(Q(z))$ for all $z\in F$,
    \item $\|\gamma(z)\|_2=1$ for all $z\in F$,
    \item $0<\textup{inf}_{\omega, z\in F} k_y^\gamma(\omega,z)$.
  \end{enumerate}
\end{Lemma}

\begin{proof}
Let $U$ and $\gamma:U\cap K_i\to\mathbb{R}^m$ be as in Lemma \ref{Lem: cont gamma}. Choose $\epsilon>0$ such that $\overline{B_\epsilon(x)}\subset U$. \\
By Lemma \ref{Lem: loc clo} there exists a closed set $\tilde{F}\subset\mathbb{R}^k$ and a neighborhood $\tilde{U}$ of $x$ such that $\tilde{F}\subset K_i$ and $\tilde{U}\cap K_i\subset\tilde{F}$.\\ 
Set $F=\overline{B_\epsilon(x)}\cap\tilde{F}$. Then $\tilde{U}\cap B_\epsilon(x)$ is a neighobrhood of $x$, and 
  \[
    (\tilde{U}\cap B_\epsilon(x)\cap K_i)\subset F. 
  \]    
Moreover, $\gamma|_F$ is well-defined, continuous and satisfies $(1)$ and $(2)$. \\
For $(3)$, note that $F\subset U\cap K_i$, so Lemma \ref{Lem: cont gamma} yields
  \[
    k_y^\gamma(\omega,z)=\langle Q(y)\gamma(\omega),\gamma(z)\rangle>0
  \]
for all $\omega, z\in F$. The map $(\omega,z)\mapsto k^\gamma_y(\omega,z)$ is continuous, and $F\times F$ is compact. Hence the infimum over $F\times F$ is attained and strictly positive, i.e.,
  \[
    \inf_{\omega, z\in F}k^\gamma_y(\omega,z)>0.
  \]
\end{proof}

Let $H$ be a Hilbert space, $\pi:C(\textup{ex}(K))\to\mathcal{B}(H)$ be a unital $*$-homomorphism. Let $A(K)$ be the continuous affine functions on $K$, and $\phi:C(\textup{ex}(K))\to\mathcal{B}(H)$ be a unital completely positive map such that $\pi(f)=\phi(f)$ for all $f\in A(K)$. For $F\subset \mathbb{R}^k$ Borel, we define $\pi(\chi_F)$ and $\phi(\chi_F)$ via the unique positive operator valued measure corresponding to $\pi$ respectively to $\phi$.\\
Recall that we assume that the spectrahedron $K$ is compact.

\begin{Lemma}\label{Lem: Main}
Assume that $\textup{Int}(K)\neq\emptyset$ and let $\pi$ and $\phi$ be as above. Let $ F\subset \partial K$ be closed, and assume that $\gamma:F\to\mathbb{R}^m$ is continuous with $\gamma(z)\in\ker(Q(z))$ and $\|\gamma(z)\|_2=1$ for all $z\in F$. Let $y\in \textup{ex}(K)\setminus F$ such that 
  \[
   \inf_{\omega,z \in F} k^\gamma_y(\omega,z)>0.
  \]
Then there exists a neighborhood $U$ of $y$ in $\mathbb{R}^k$ such that 
  \[
    \pi(\chi_F)\phi(\chi_U)\pi(\chi_F)=0.
  \]
  \end{Lemma}
   
\begin{proof}
The proof is divided into five steps:\\
We first cover $F$ by finitely many sets $F_1,\dots, F_n$ on which the diagonal values $k_z^\gamma(x_i,x_i)$ are uniformly small (Step~1), and encode the kernel on this finite set in the matrix-valued map $A(z)=(k_z^\gamma(x_i,x_j))_{1\le i,j\le n}$ (Step~2).\\
In Step~3 we use the standing assumption and continuity to decompose $A=A^++A^-$, where the negative part $A^-$ is small in norm. \\
Next, we choose a neighborhood $U$ of $y$ via a suitable convex combination with a point in $\mathrm{Int}(K)$ (Step~4).\\
Finally, in Step~5 we apply Perron--Frobenius theory to produce a positive semidefinite matrix $B$ such that 
  \[
    (\chi_{U\cap K})_{1\le i,j\le n}\le \frac{1}{1-s}\,B\circ A|_K. 
  \]    
Applying $\phi$ and compressing by $(\pi(\chi_{F_1}),\dots,\pi(\chi_{F_n}))$ yields an estimate of the form $\|\pi(\chi_F)\phi(\chi_U)\pi(\chi_F)\|\le C\,\varepsilon$; letting $\varepsilon\to0$ gives the desired conclusion.

\medskip
So let us start by fixing $\epsilon>0$.\\
\textbf{Step 1}.
Retaining the notation introduced in the statement of the theorem, set for $x\in F$
  \[
    U_x=\{z\in K;\ k^\gamma_z(x,x)<\epsilon\}.
  \]
Each $U_x$ is open in $K$ since $z\mapsto k_z^\gamma(x,x)=\gamma(x)^*Q(z)\gamma(x)$ is continuous. Moreover, $F\subset \bigcup_{x\in F}U_x$ because $k^\gamma_x(x,x)=0$ for $x\in F$. As $K$ is compact and $F$ closed, also $F$ is compact, so there exist $x_1,\dots,x_n\in F$ such that
   \[
     F\subset\bigcup_{i=1}^n U_{x_i}.
   \]
Define Borel sets
  \[
    F_1=F\cap U_{x_1},\qquad F_i=F\cap \left(U_{x_i}\setminus\bigcup_{j<i} U_{x_j}\right)\qquad 2\le i\le n.
  \]
Then $F_1,\dots,F_n$ are disjoint Borel sets, $F=\bigcup_{i=1}^nF_i$ and
   \[
     \sup_{z\in F_i} k^\gamma_z(x_i,x_i)\le\epsilon, \qquad 1\le i\le n.
    \]
\textbf{Step 2}.
Consider the affine matrix-valued map
  \[
    A:\mathbb{R}^k\to M_n, z\mapsto (k_{z}^\gamma(x_i,x_j))_{1\le i,j\le n}.
  \]
Since $k^\gamma_z(\cdot,\cdot)$ is a positive semidefinite kernel on $F\times F$ for all $z\in K$, the restriction $A|_K$ is a positive element of $M_n(A(K))$.\\
\textbf{Step 3}.
By assumption and since $\|\gamma(z)\|_2=1$ for all $z\in F$, 
     \[
          0<\inf_{\omega, z\in F}k^\gamma_y(\omega,z)\le\sup_{\omega, z\in F}k^\gamma_y(\omega,z)\le\|Q(y)\|<\infty.
     \]
By continuity of $\gamma$ and compactness of $F$, there exist a neighborhood $V$ of $y$ and numbers $0<c\le b$ such that for all $z\in V$:
      \[
          0<c\le\inf_{\omega, x\in F}k^\gamma_z(\omega,x)\le\sup_{\omega, x\in F}k^\gamma_z(\omega,x)\le b<\infty.
     \]
Since $Q(\cdot)$ is continuous and $Q(y)\ge0$, it follows that for every $\alpha>0$ there exists a neighborhood $\tilde{V}_\alpha$ of $y$ such that for all $z\in\tilde{V}_\alpha$:
  \[
    \left\|\sum_{\lambda\in\sigma(Q(z)),\ \lambda<0}\lambda P_\lambda(z)\right\|<\alpha,
  \]
where $P_\lambda(z)$ denotes the spectral projection of $Q(z)$ corresponding to $\lambda$. For later computational reasons, we now fix 
  \[
    \alpha:=\tfrac{c^3}{16(b+c)^2}
  \]
and write $\tilde{V}=\tilde{V}_\alpha$. Note that since $\|\gamma(\cdot)\|_2=1$, we have for all $z\in \tilde{V}$ and $x,\omega\in F$:
   \begin{equation}
  \label{Eq: cont gamma 3}
   \left|\gamma(\omega)^*\left[\sum_{\lambda\in\sigma(Q(z)), \lambda<0}\lambda P_\lambda(z)\right]\gamma(x)\right|\le\tfrac{c^3}{16(b+c)^2}=\alpha.
   \end{equation}
For $z\in\tilde{V}$ write the spectral decomposition
  \[
    X(z)=\sum_{\lambda\in\sigma(Q(z)), \lambda<0} \lambda P_\lambda(z), \qquad  Y(z)=\sum_{\lambda\in\sigma(Q(z)), \lambda>0} \lambda P_\lambda(z),
  \]
so that $X(z)\le0\le Y(z)$ and $Q(z)=X(z)+Y(z)$. Let $\Gamma_n:\mathbb{R}^n\to\mathbb{R}^{mn}$ be the isometry
  \[
   \Gamma_n(t_1,\dots,t_n)=(t_1\gamma(x_1),\dots,t_n\gamma(x_n)).
  \]
Then for $z\in\tilde{V}$ and $(\textbf{1})=(1)_{1\le i,j\le n}$,
  \begin{equation*}
    \begin{split}
      A(z)&=(\gamma(x_j)^*Q(z)\gamma(x_i))_{1\le i,j\le n}=\Gamma_n^*[(\textbf{1})\otimes Q(z)]\Gamma_n\\
       &=\Gamma_n^*\left[(\textbf{1})\otimes X(z)\right]\Gamma_n+\Gamma_n^*\left[(\textbf{1})\otimes Y(z)\right]\Gamma_n\\
       &=:A^-(z)+A^+(z)
     \end{split}
  \end{equation*}
For $z\in V\cap\tilde{V}$, using the bounds on $k^\gamma_z$ and \cref{Eq: cont gamma 3}, we obtain for all $i,j$:
   \[
     \begin{split}
       (A^+(z))_{i,j}=(\Gamma_n^*\left[(\textbf{1})\otimes Y(z)\right]\Gamma_n)_{i,j}&=\gamma(x_j)^*Q(z)\gamma(x_i)-\left(\Gamma_n^*\left[(\textbf{1})\otimes X(z)\right]\Gamma_n\right)_{i,j}\\
       &=k^\gamma_z(x_i,x_j)-\gamma(x_j)^*\left[\sum_{\lambda\in\sigma(Q(z)), \lambda<0} \lambda P_\lambda\right]\gamma(x_i)\\
       &\ge c-\alpha>0,
    \end{split}
  \]
since $\alpha=\tfrac{c^3}{16(b+c)^2}\le \tfrac{c}{2}$, and similarly
  \[
    (A^+(z))_{i,j}=(\Gamma_n^*\left[(\textbf{1})\otimes Y(z)\right]\Gamma_n)_{i,j}\le b+\alpha
  \]
Thus the entries of $A^+(z)=\Gamma_n^*\left[(\textbf{1})\otimes Y(\omega)\right]\Gamma_n$ are between $c-\alpha$ and $b+\alpha$, in particular, strictly positive, and so there exists a Perron eigenvalue $\lambda_z$ and normalized eigenvector $p_z$ of $A^+(z)$ that, by Lemma \ref{Lem: Perron Bound}, satisfy
  \begin{equation}\label{Eq: cont gamma 4}
    (\lambda_z p_z p_z^*)_{i,j}\ge\tfrac{(c-\alpha)^3}{(b+\alpha)^2}=:\tilde{c}.
  \end{equation}
Since $\alpha\le c/2$, we have
  \[
    \tilde{c}\ge\tfrac{(c/2)^3}{(b+c)^2}=\tfrac{c^3}{8(b+c)^2}=2\alpha,
    \]
and therefore, combining \cref{Eq: cont gamma 3} and \cref{Eq: cont gamma 4}:
    \begin{equation}\label{Eq: cont gamma 5}
      (\lambda_z p_z p_z^*+A^-(z))_{i,j}=\left(\lambda_z p_z p_z^*+\Gamma_n^*\left[(\textbf{1})\otimes X(z)\right]\Gamma_n\right)_{i,j}\ge\tilde{c}-\alpha\ge\alpha
  \end{equation}
for all $z\in V\cap\tilde{V}$.\\
\textbf{Step 4}.
Set $\tilde{U}=V\cap\tilde{V}$. Choose $\omega_1\in \tilde{U}\cap \textup{Int}(K)$, which is possible since $\textup{Int}(K)\neq\emptyset$ by assumption and $\tilde{U}$ is a neighborhood of $y$. Since $y\in\partial K$, we can pick $\omega_2\in\tilde{U}\setminus K$ and $s\in(0,1)$ such that
  \[
    y=s\omega_1+(1-s)\omega_2.
  \]
Choose $\delta>0$ such that $B_{\delta}(\omega_1)\subset \textup{Int}(K)$ and $B_{\delta}(y)\subset\tilde{U}$, and define
  \begin{equation}\label{Eq: cont gamma 6}
    U= \{s\omega+(1-s)\omega_2;\ \omega\in B_{\delta}(\omega_1)\}=y+sB_{\delta}(0)=B_{s\delta}(y)\subset\tilde{U}.
  \end{equation}
Then $U$ is an open neighborhood of $y$.\\
\textbf{Step 5}.
Recall that we decomposed $A(\omega_2)$ into (we drop the notational dependence on $\omega_2$) a positive operator $A^+$ and a negative operator $A^-$. Furthermore $A^+$ only has strictly positive entries and admits a Perron eigenvalue $\lambda$ with normalized eigenvector $p$ such that for all $i,j$:
   \[
     \left(\lambda pp^*+ A^-\right)_{i,j}\ge\alpha.
  \]
Hence $\lambda pp^*+A^-$ has strictly positive entries and exactly one strictly positive eigenvalue, since 
  \[
    \langle (\lambda pp^*+A^- )x,x\rangle\le 0
  \]
for all $x\perp p$. Therefore, by \cite[Theorem 2.7]{Rea}, the Hadamard inverse $B$ of $\lambda pp^*+A^-$ is positive semidefinite. Thus, $B\circ A(z)\ge0$ for all $z\in K$. Furthermore, since $A^+-\lambda pp^*\ge0$ and so $B\circ(A^+-\lambda pp^*)\ge0$,
  \[
    (\textbf{1})=B\circ (\lambda pp^*+ A^-)\le B\circ(A^++A^-)= B\circ A(\omega_2).
  \]
Now let $x\in U$. By \cref{Eq: cont gamma 6} we can write $x=s\omega+(1-s)\omega_2$ for some $\omega\in B_{\delta}(\omega_1)\subset\textup{Int}(K)$. Using $B\circ A(\omega)\ge0$ we obtain
  \[
    (1-s)(\textbf{1})\le (1-s)B\circ A(\omega_2)+sB\circ A(\omega)=B\circ A(x).
  \]
 Consequently, 
  \[
    (\chi_{U\cap K})_{1\le i,j\le n}\le\tfrac{1}{1-s} B\circ A|_K
  \]
Applying the u.c.p.\ map $\phi$, using that $\phi(\chi_{U\cap K})=\phi(\chi_U)$, and compressing by $\pi(\chi_F)=\sum_{i=1}^n \pi(\chi_{F_i})$, we obtain
  \[
    \begin{split}
      \pi(\chi_F)\phi(\chi_U)\pi(\chi_F)&=(\pi(\chi_{F_1}),\dots,\pi(\chi_{F_n}))\, \phi((\chi_U)_{1\le i,j\le n}) (\pi(\chi_{F_1}),\dots,\pi(\chi_{F_n}))^*\\
      &\leq(\pi(\chi_{F_1}),\dots,\pi(\chi_{F_n}))\, \phi\!\left(\frac{1}{1-s}B\circ A|_K\right) (\pi(\chi_{F_1}),\dots,\pi(\chi_{F_n}))^*.
    \end{split}
  \]
Since $\phi=\pi$ on $A(K)$, the right-hand side equals
  \[
    \frac{1}{1-s}\,(\pi(\chi_{F_1}),\dots,\pi(\chi_{F_n}))\,
    \pi(B\circ A|_K)\,
    (\pi(\chi_{F_1}),\dots,\pi(\chi_{F_n}))^*
    =
    \frac{1}{1-s}\bigoplus_{i=1}^n \pi\bigl(\chi_{F_i}(B\circ A)_{i,i}\bigr),
  \]
because the sets $F_i$ are disjoint and $\pi$ is a $*$-homomorphism. Taking norms yields
  \[
    \|\pi(\chi_F)\phi(\chi_U)\pi(\chi_F)\|
    \le \frac{1}{1-s}\max_{1\le i\le n}\|\chi_{F_i}(B\circ A)_{i,i}\|_\infty.
  \]
Since $B$ is the Hadamard inverse of $\lambda pp^*+ A^-$ and by \cref{Eq: cont gamma 5} $(\lambda pp^*+ A^-)_{i,i}\ge\alpha$, we have $B_{i,i}\le \alpha^{-1}$, hence 
  \[
    \|\chi_{F_i} (B\circ A)_{i,i}\|_\infty\le \alpha^{-1}\|\chi_{F_i}A_{i,i}\|_\infty.
  \]
But $A_{i,i}(z)=k_z^\gamma(x_i,x_i)$ and by construction $0<\sup_{z\in F_i}k_z^\gamma(x_i,x_i)\le\epsilon$. Therefore
  \[
    \begin{split}       
      \| \pi(\chi_F)\phi(\chi_U)\pi(\chi_F)\|&\le \tfrac{1}{1-s}\max_{1\le i\le n}\|\chi_{F_i} (B\circ A)_{i,i}\|_\infty\\
      &\le \tfrac{1}{\alpha(1-s)}\max_{1\le i\le n}\|\chi_{F_i} (A)_{i,i}\|_\infty\\
      &\le \tfrac{1}{\alpha(1-s)}\epsilon.
    \end{split}
  \]
Since $b, c$, and thus also $\alpha$, $s$ and $U$ are independent of $\epsilon$, we conclude that
  \[
    \| \pi(\chi_F)\phi(\chi_U)\pi(\chi_F)\|=0.
  \]
\end{proof}

The final step in the proof of the hyperrigidity of $A(K)$ is purely measure-theoretic and does not use any spectrahedral
structure. In fact, the underlying separation principle appears to be of independent interest and may well prove useful for hyperrigidity problems beyond the present setting. For this reason, we state the last step in a general form.

\begin{Theorem}\label{Thm: final step}
Let $X$ be a compact metrizable space and let $H$ be a Hilbert space.
Let $\pi:C(X)\to\mathcal{B}(H)$ be a unital $*$-homomorphism and let $\phi:C(X)\to\mathcal{B}(H)$ be a u.c.p.\ map.
Assume that for every pair of distinct points $x,y\in X$ there exist disjoint open neighborhoods $U$ of $x$
and $V$ of $y$ such that
\[
  \pi(\chi_U)\,\phi(\chi_V)\,\pi(\chi_U)=0.
\]
Then $\phi=\pi$ on $C(X)$.
\end{Theorem}

\begin{proof}
With the above notations, we start by showing that for any disjoint open sets $U, V\subset X$ one has
\begin{equation}\label{Eq: disjoint vanishing}
  \phi(\chi_V)\,\pi(\chi_U)=0
  \qquad\text{and}\qquad
  \pi(\chi_U)\,\phi(\chi_V)=0.
\end{equation}
Fix disjoint open $U, V\subset X$, and let $M\subset U$ be compact. For each $x\in M$ and $y\in V$ choose disjoint open neighborhoods $U_{x,y}$ of $x$ and $V_{x,y}$ of $y$ such that $V_{x,y}\subset V$ and
  \[
    \pi(\chi_{U_{x,y}})\phi(\chi_{V_{x,y}})\pi(\chi_{U_{x,y}})=0.
  \]
Fix $y\in V$. By compactness of $M$ we may cover $M$ by finitely many sets $U_{x_1,y},\dots, U_{x_n,y}$ and set
  \[
    U_y=\bigcup_{i=1}^n U_{x_i,y},\qquad V_y=\bigcap_{i=1}^n V_{x_i,y}.
  \]
Then $M\subset U_y$ and $y\in V_{y}$. Moreover, for each $i$ we have
  \[
    \pi(\chi_{U_{x_i,y}})\phi(\chi_{V_y})\pi(\chi_{U_{x_i,y}})=0,
  \]
since $V_y\subset V_{x_i,y}$. Using \cite[Corollary 1.2]{Br} we obtain
  \[
    0\le\pi(\chi_M)\phi(\chi_{V_y})\pi(\chi_M)\le\pi(\chi_{U_y})\phi(\chi_{V_y})\pi(\chi_{U_y})=0.
  \]
As $\phi(\chi_{V_y})\ge0$, this implies
  \[
  \phi(\chi_{V_y})\pi(\chi_M)=0.
  \]
Now $\{V_y;\ y\in V\}$ is an open cover of $V$. Since $X$ is metrizable, it is Lindel\"of, so there exist $(y_j)_{j=1}^\infty$ in $V$ with $V\subset\bigcup_{j\ge1}V_{y_j}$. Since by choice $V_{y_j}\subset V$ for all $j$, we even have $V=\bigcup_{j\ge1}V_{y_j}$. Let $W_N=\bigcup_{j=1}^N V_{y_j}$. Then $\phi(\chi_{W_N})\to \phi(\chi_V)$ in the WOT and
  \[
    \pi(\chi_M)\phi(\chi_{W_N})\pi(\chi_M)\le\sum_{j=1}^N\pi(\chi_M)\phi(\chi_{V_{y_j}})\pi(\chi_M)=0.
  \]
Thus $\phi(\chi_{W_N})\pi(\chi_M)=0$ and
  \[
    0=\lim_{N\to\infty}\phi(\chi_{W_N})\pi(\chi_M)=\phi(\chi_V)\pi(\chi_M).
  \]
Finally, choose an increasing sequence of compact sets $M_\ell\subset U$ with $\bigcup_\ell M_\ell=U$ (possible since $X$ is compact metrizable). Then
$\pi(\chi_{M_\ell})\to \pi(\chi_U)$ strongly, so passing to the limit yields
  \[
    \phi(\chi_V)\pi(\chi_U)=0. 
  \]    
This proves the first identity in \eqref{Eq: disjoint vanishing}.
The second follows by taking adjoints.\\
We now show that $\phi(\chi_U)=\pi(\chi_U)$ for every open set $U\subset X$.\\
Since $X$ is metrizable, pick open sets $U_1,U_2,\dots\subset U$ with $\overline{U_n}\subset U$ and
$U=\bigcup_{n\ge 1}U_n$. Set $V_n=\bigcup_{j=1}^n U_j$ and $\widetilde V_n=X\setminus \overline{V_n}$.
Then $V_n$ and $\widetilde V_n$ are disjoint open sets, $V_n\uparrow U$, and $\widetilde V_n\downarrow U^c$.
By \eqref{Eq: disjoint vanishing},
  \[
    \phi(\chi_{\widetilde V_n})\,\pi(\chi_{V_n})=0
    \quad\text{and}\quad
    \phi(\chi_{V_n})\,\pi(\chi_{\widetilde V_n})=0
    \qquad (n\ge 1).
  \]
Since
  \[
    \pi(\chi_{\widetilde V_n})\to\pi(\chi_{U^c})
    \quad\text{and}\quad
    \pi(\chi_{V_n})\to\pi(\chi_V)
  \]
in the SOT, and 
  \[
    \phi(\chi_{\widetilde V_n})\to\phi(\chi_{U^c})
    \quad\text{and}\quad
    \phi(\chi_{V_n})\to\phi(\chi_V)
  \]
in the WOT, and using that the product of a WOT convergent sequence with a SOT convergent sequence converges WOT, we get
  \[
    \phi(\chi_{U^c})\pi(\chi_U)=\lim \phi(\chi_{\widetilde V_j})\pi(\chi_{V_j})=0
  \]
and
  \[
    \phi(\chi_U)\pi(\chi_{U^c})=\lim \phi(\chi_{V_j})\pi(\chi_{\widetilde V_j})=0.
  \]
Now
  \[
    \phi(\chi_U)=\phi(\chi_U)\pi(\chi_U)+\phi(\chi_U)\pi(\chi_{U^c})+\phi(\chi_{U^c})\pi(\chi_U)=\pi(\chi_U).
  \]
Finally, it is straightforward to check that 
  \[
     \{F\subset X;\ F\ \textup{Borel},\ \phi(\chi_F)=\pi(\chi_F)\}
  \]
is a Dynkin-system that contains all open sets. By the Dynkin $\pi-\lambda$ theorem, we obtain that $\pi(\chi_F)=\phi(\chi_F)$ for every Borel set $F$. Thus $\pi=\phi$.
\end{proof}

\begin{Theorem}\label{Thm: main}
Let $K\subset\mathbb{R}^k$ be a compact spectrahedron such that $\textup{ex}(K)$ is closed. Then $A(K)$ is hyperrigid in $C(\textup{ex}(K))$.
\end{Theorem}

\begin{proof}
Let $K\subset\mathbb{R}^k$ be a compact spectrahedron such that $\textup{ex}(K)$ is closed. Since $K$ is compact and $\textup{ex}(K)$ is closed in $K$, the space $\textup{ex}(K)$ is compact and metrizable.\\
Let $\pi:C(\textup{ex}(K))\to\mathcal{B}(H)$ be a unital $*$-homomorphism and let $\phi:C(\textup{ex}(K))\to\mathcal{B}(H)$ be a u.c.p.\ map
such that $\phi=\pi$ on $A(K)$. We show that $\phi=\pi$ on $C(\textup{ex}(K))$.\\
Following the discussion at the beginning of this section, we know that there exists a symmetric pencil $Q$ such that
  \[
    K=\{z\in\mathbb{R}^k;\ Q(z)\ge0\}.
  \]
\textbf{Case 1: $\textup{Int}(K)\neq\emptyset$.}\\ 
By Theorem \ref{Thm: final step}, it suffices to prove that for every $x\neq y\in\textup{ex}(K)$ there are neighborhoods $U$ of $x$ and $V$ of $y$ such that 
  \[
    \pi(\chi_U)\phi(\chi_V)\pi(\chi_U)=0.
  \]
Fix $x\neq y\in\textup{ex}(K)$. Since $x,y\in\mathbb{R}^k$, we find open sets $U$ and $V$ with $x\in U$, $y\in V$, $U\cap\overline{V}=\emptyset$ and $\overline{V}$ compact.\\
Fix $j\in\{1,\dots,m\}$ and let $z\in U\cap K_j$ and $\omega\in\overline{V}$. Applying Lemma \ref{Lem: cont gamma 2} and Lemma \ref{Lem: Main}, we obtain an open neighborhood $U_{z,\omega}$ of $z$ such that $U_{z, \omega}\subset U$ and an open neighborhood $V_{z,\omega}$ of $\omega$ such that 
  \[
    \pi(\chi_{U_{z,\omega}\cap K_j})\phi(\chi_{V_{z,\omega}})\pi(\chi_{U_{z,\omega}\cap K_j})=0.
  \]
(If $z\notin\textup{ex}(K)$ or $\omega\notin\textup{ex}(K)$, pick open neighborhoods $z\in U_{z,\omega}\subset U$ and $\omega\in V_{z,\omega}$ with $U_{z,\omega}\cap \textup{ex}(K)=\emptyset$ and/or $V_{z,\omega}\cap\textup{ex}(K)=\emptyset$ accordingly; such neighborhoods exist because $\textup{ex}(K)$ is closed.) \\
Now fix $z\in U\cap K_j$. Since $(V_{z,\omega})_{\omega\in\overline{V}}$ is an open covering of the compact set $\overline{V}$, there exists a finite subcover $(V_{z,\omega_i})_{i=1}^n$. Set
  \[
    U_z=\bigcap_{i=1}^nU_{z,\omega_i}\qquad \textup{and}\qquad V_z=\bigcup_{i=1}^nV_{z,\omega_i}.
  \]
Then $U_z$ is an open neighborhood of $z$ contained in $U$, and $V\subset V_z$. Hence
  \[
    \pi(\chi_{U_z\cap K_j})\phi(\chi_V)\pi(\chi_{U_z\cap K_j})=0 \qquad \textup{for all}\ z\in U\cap K_j.
  \]
For each $j$, the sets $(U_z)_{z\in U\cap K_j}$ form an open cover of $U\cap K_j$. As $U\cap K_j$ is Lindel\"of, we can choose points $z^j_i\in U\cap K_j$ such that 
  \[
    U\cap K_j\subset\bigcup_{i=1}^\infty U_{z^j_i}.
  \]
Define the increasing open sets
  \[
    U^j_N=K_j\cap\left(\bigcup_{i=1}^N U_{z^{j}_i}, \qquad N\in\mathbb{N}\right).
  \]
By \cite[Corollary 1.2]{Br} we have that for all $N\in\mathbb{N}$
  \[
    \pi(\chi_{U^j_N})\phi(\chi_V)\pi(\chi_{U^j_N})=0.
  \]
Finally, $\pi(\chi_{U^j_N})\to\pi(\chi_{U\cap K_j})$ SOT, hence taking SOT-limits yields
  \[
    \pi(\chi_{U\cap K_j})\phi(\chi_V)\pi(\chi_{U\cap K_j})=0,
  \]
and once more with \cite[Corollary 1.2]{Br}, we get that
  \[
    \pi(\chi_{U})\phi(\chi_V)\pi(\chi_{U})\le\sum_{j=1}^m
    \pi(\chi_{U\cap K_j})\phi(\chi_V)\pi(\chi_{U\cap K_j})=0.
  \]
Thus the hypothesis of Theorem \ref{Thm: final step} holds, and therefore $\pi=\phi$. This proves the hyperrigidity of $A(K)$ in $C(\textup{ex}(K))$ when $\textup{Int}(K)\neq\emptyset$.\\

\textbf{Case 2: $\textup{Int}(K)=\emptyset$.}\\
Denote the affine hull of a set by $\textup{aff}(\cdot)$. Let $d=\dim(\textup{aff}(K))<k$ and choose an affine homeomorphism $L:\mathbb{R}^d\to \textup{aff}(K)$.
Set $\widetilde K=L^{-1}(K)$ and $\widetilde Q(w)=Q(L(w))$. Then
\[
  \widetilde K=\{w\in\mathbb{R}^d;\ \widetilde Q(w)\ge 0\},
\]
so $\widetilde K$ is a compact spectrahedron in $\mathbb{R}^d$, and $\textup{aff}(\widetilde K)=\mathbb{R}^d$.
In particular, $\textup{Int}(\widetilde K)\neq\emptyset$ unless $K$ is a singleton.\\
Moreover, $L$ induces a homeomorphism $\textup{ex}(\widetilde K)\cong \textup{ex}(K)$ (affine bijections preserve extreme points), and composition with $L$ yields a unital complete order isomorphism between $A(K)$ and $A(\widetilde K)$.
Thus hyperrigidity of $A(K)$ in $C(\textup{ex}(K))$ is equivalent to hyperrigidity of $A(\widetilde K)$ in $C(\textup{ex}(\widetilde K))$.\\
If $K$ is not a singleton, then $\textup{Int}(\widetilde K)\neq\emptyset$ and Case~1 applied to $\widetilde K$ yields the desired conclusion for $K$.\\
If $K$ is a singleton, then $A(K)=C(\textup{ex}(K))$ and hyperrigidity is trivial.
\end{proof}


\bibliography{Bib_Hyper_Spectrahedra} 
\bibliographystyle{plain}
Technion, Haifa, Israel\\
\textit{Email address:} scherer@math.uni-sb.de

\end{document}